\documentclass[a4paper,12pt,twoside]{article}
\usepackage{amsthm,amssymb,amsmath,amscd}
\usepackage{pb-diagram}
\usepackage{latexsym,amsmath,amsthm,amssymb,amscd}
\newtheorem{thm}{Theorem}[section]
\newtheorem{lem}[thm]{Lemma}

\newtheorem{proposition}[thm]{Proposition}
\newtheorem{example}[thm]{Example}
\newtheorem{defn}[thm]{Definition}

\usepackage[arrow,frame,matrix]{xy}
\input xy

\begin{document}
\def\Ad{{\rm Ad}}
\def\diag{{\rm diag}}
\def\End{{\rm End}}
\def\Fr{{\rm Fr}}
\def\Gal{{\rm Gal}}
\def\GL{{\rm GL}}
\def\Aff{{\rm Aff}}
\def\Id{{\rm I}}
\def\Norm{{\rm Norm}}
\def\Nrd{{\rm Nrd}}
\def\Stab{{\rm Stab}}
\def\Id{{\rm Id}}
\def\fa{{\mathfrak a}}
\def\haut{{\rm ht}}
\def\imag{{\rm im}}
\def\integers{{\mathfrak o}}
\def\fR{{\mathfrak R}}
\def\LieN{{\rm n}}
\def\re{{\rm Re}}
\def\Rep{{\rm Rep}}
\def\WDRep{{\rm WDRep}}
\def\rec{{\rm rec}}
\def\Hom{{\rm Hom}}
\def\Ind{{\rm Ind}}
\def\cInd{{\rm c\!\!-\!\!Ind}}
\def\pr{{\rm pr}}
\def\supp{{\rm supp}}
\def\Mod{{\rm Mod}}
\def\ii{{\rm i}}
\def\rr{{\rm r}}
\def\Pbar{{\rm \overline P}}
\def\Nbar{{\rm \overline N}}
\def\sw{{\rm sw}}
\def\Sw{{\rm Sw}}
\def\Ar{{\rm A}}
\def\Sym{{\rm Sym}}
\def\Cusp{{\rm Cusp}}
\def\Irr{{\rm Irr}}
\def\Irrt{{{\rm Irr}^{\rm t}}}
\def\Psit{{{\Psi}^{\rm t}}}
\def\temp{{{\rm t}}}
\def\simple{{{\rm s}}}
\def\Irru{{{\rm Irr}^{\rm u}}}
\def\mes{{\rm mes}}
\def\Tor{{\rm Tor}}
\def\Nor{{\rm N}}
\def\ie{{\it i.e.,\,}}
\def\tr{{\rm tr}}
\def\Omegat{{{\Omega}^{\rm t}}}
\def\cD{{\mathcal D}}
\def\cE{{\mathcal E}}
\def\cH{{\mathcal H}}
\def\cL{{\mathcal L}}
\def\cO{{\mathcal O}}
\def\sA{{\mathcal A}}
\def\sV{{\mathcal V}}
\def\spin{{\rm sp}}
\def\SL{{\rm SL}}
\def\Sp{{\rm Sp}}
\def\SU{{\rm SU}}
\def\U{{\rm U}}
\def\red{{\rm nd}}
\def\CC{{\mathbb C}}
\def\QQ{{\mathbb Q}}
\def\RR{{\mathbb R}}
\def\ZZ{{\mathbb Z}}
\def\Afr{{\mathfrak A}}
\def\Bfr{{\mathfrak B}}
\def\Cfr{{\mathfrak C}}
\def\Ofr{{\mathfrak O}}
\def\av{{|\textrm{ }|}}
\def\FF{{F^\times}}
\def\F{{\mathbb F}}
\def\Z{{\mathbb Z}}
\def\N{{\mathbb N}}
\def\ZZ{{\mathbb{Z}^\times}}
\def\R{{\mathbb{R}}}
\def\T{{\mathbb T}}
\def\RR{{\mathbb{R}^\times}}
\def\C{{\mathbb C}}
\def\CC{{\mathbb{C}^\times}}
\def\slashi#1{\rlap{\sl/}#1}

\title{Functoriality and $K$-theory for $\GL_n(\mathbb{R})$ }
\author{Sergio Mendes and Roger Plymen}
\date{}
\maketitle

\begin{abstract}   We investigate base change and automorphic induction $\mathbb{C}/\mathbb{R}$
at the level of $K$-theory for the general linear group
$\GL_n(\mathbb{R})$. In the course of this study, we compute in
detail the $C^*$-algebra $K$-theory of this disconnected group. We
investigate the interaction of base change with the Baum-Connes
correspondence for $\GL_n(\R)$ and $\GL_n(\C)$. This article is
the archimedean companion of our previous article \cite{MP}.
\end{abstract}



\section{Introduction}

In the general theory of automorphic forms, an important role is
played by \emph{base change} and \emph{automorphic induction}, two
examples of the principle of functoriality in the Langlands program \cite{BS}. Base change
and automorphic induction have a global aspect and
a local aspect \cite{AC}\cite{He}.  In this article, we  focus on the
archimedean case of base change and automorphic induction for the general linear group
$\GL(n,\R)$, and we investigate these aspects of functoriality at the
level of $K$-theory.

For $\GL_n(\R)$ and $\GL_n(\C)$ we have the Langlands
classification and the associated $L$-parameters \cite{K}. We
recall that the domain of an $L$-parameter of $\GL_n(F)$ over an
archimedean field $F$ is the Weil group $W_F$. The Weil groups are
given by
\[
W_{\mathbb{C}}=\mathbb{C}^{\times}
\]
and
\[
W_{\mathbb{R}}=\langle j\rangle\mathbb{C}^{\times}
\]
where $j^2=-1 \in\mathbb{C}^{\times}$, $jc=\overline{c}j$ for all $c\in\mathbb{C}^{\times}$. Base change is defined by restriction of $L$-parameter from $W_{\R}$ to $W_{\C}$.

An $L$-parameter $\phi$ is \emph{tempered} if $\phi(W_F)$ is
bounded. Base change therefore determines a map of tempered duals.

In this article, we investigate the interaction of base change
with the Baum-Connes correspondence for $\GL_n(\R)$ and
$\GL_n(\C)$.

Let $F$ denote $\R$ or $\C$ and let $G = G(F) = \GL_n(F)$. Let
$C^{*}_{r}(G)$ denote the reduced $C^*$-algebra of $G$. The
Baum-Connes correspondence is a canonical isomorphism
\cite{BCH}\cite{CEN}\cite{La}
$$\mu_{F} : K^{G(F)}_{*}(\underline{E}G(F)) \rightarrow K_{*}C^{*}_{r}(G(F))$$
where $\underline{E}G(F)$ is a universal example for the action of
$G(F)$.

The noncommutative space $C^{*}_{r}(G(F))$ is strongly Morita
equivalent to the commutative $C^*$-algebra
$C_{0}(\mathcal{A}^t_n(F))$ where $\mathcal{A}^t_n(F)$ denotes the
tempered dual of $G(F)$, see \cite[\S 1.2]{Pl1}\cite{PP}. As a
consequence of this, we have
$$K_{*}C^{*}_{r}(G(F))\cong K^{*}\mathcal{A}^t_n(F).$$
This leads to the following formulation of the Baum-Connes
correspondence:
$$ K^{G(F)}_{*}(\underline{E}G(F))\cong K^{*}\mathcal{A}^t_n(F).$$

Base change and automorphic induction $\C/\R$ determine maps
$$\mathcal{BC}_{\C/\R} : \mathcal{A}^t_n(\R) \rightarrow \mathcal{A}^t_n(\C)$$
and
$$\mathcal{AI}_{\C/\R} : \mathcal{A}^t_n(\C) \rightarrow \mathcal{A}^t_{2n}(\R).$$
This leads to the following diagrams

\[\CD
K^{G(\C)}_{*}(\underline{E}G(\C)) @> {\mu_{\C}}> {}>K^{*}\mathcal{A}^t_n(\C))\\
@V{ }VV @VV{\mathcal{BC}_{\C/\R}^{*}} V \\
K^{G(\R)}_{*}(\underline{E}G(\R)) @>{}> {\mu_{\R}}>
K^{*}\mathcal{A}^t_n(\R).
\endCD
\]
and
\[\CD
K^{G(\R)}_{*}(\underline{E}G(\R)) @> {\mu_{\R}}> {}>K^{*}\mathcal{A}^t_{2n}(\R))\\
@V{ }VV @VV{\mathcal{AI}_{\C/\R}^{*}} V \\
K^{G(\C)}_{*}(\underline{E}G(\C)) @>{}> {\mu_{\C}}>
K^{*}\mathcal{A}^t_n(\C).
\endCD
\]
where the left-hand vertical maps are the unique maps which make the
diagrams commutative.

In section $2$ we describe the tempered dual
$\mathcal{A}^{t}_{n}(F)$ as a locally compact Hausdorff space.

In section $3$ we compute the $K$-theory for the reduced
$C^*$-algebra of $\GL(n,\mathbb{R})$.
The real reductive Lie group $\GL(n,\mathbb{R})$ is not connected.
If $n$ is even our formulas show that we always have non-trivial
$K^0$ and $K^1$. We also recall the $K$-theory for the reduced
$C^*$-algebra of the complex reductive group $\GL(n,\mathbb{C})$,
see \cite{PP}. In section $4$ we recall the Langlands parameters for $\GL(n)$ over archimedean local fields, see \cite{K}.
In section $5$ we compute the base change map $\mathcal{BC}:\mathcal{A}^{t}_{n}(\mathbb{R}) \to
\mathcal{A}^{t}_{n}(\mathbb{C})$ and prove that $\mathcal{BC}$ is a
continuous proper map. At the level of $K$-theory, base change is
the zero map for $n>1$ (Theorem \ref{main result archimedean base
change n>1}) and is nontrivial for $n=1$ (Theorem \ref{main result archimedean base change n=1}).
In section $6$, we compute the automorphic induction map $\mathcal{AI}:\mathcal{A}^{t}_{n}(\mathbb{C}) \to
\mathcal{A}^{t}_{2n}(\mathbb{R})$. Contrary to base change, at the level of $K$-theory, automorphic induction is nontrivial for every $n$
(Theorem \ref{main result automorphic induction n}). In section $7$, where we study the case $n=1$, base change for
$K^1$ creates a map
\[
\mathcal{R}(\U(1))\longrightarrow\mathcal{R}(\mathbb{Z}/2\mathbb{Z})
\]
where $\mathcal{R}(\U(1))$ is the representation ring of the
circle group $\U(1)$ and $\mathcal{R}(\mathbb{Z}/2\mathbb{Z})$ is
the representation ring of the group $\mathbb{Z}/2\mathbb{Z}$.
This map sends the trivial character of $\U(1)$ to
$1\oplus\varepsilon$, where $\varepsilon$ is the nontrivial
character of $\mathbb{Z}/2\mathbb{Z}$, and sends all the other
characters of $\U(1)$ to zero.

This map has an interpretation in terms of $K$-cycles. The
$K$-cycle \[(C_0(\R), L^2(\R),id/dx)\] is equivariant with respect
to $\CC$ and $\RR$, and therefore determines a class $\slashi
\partial_{\C} \in K_1^{\CC}(\underline{E}\CC)$ and a class
$\slashi \partial_{\R} \in K_1^{\RR}(\underline{E}\RR)$.  On the
left-hand-side of the Baum-Connes correspondence, base change in
dimension $1$ admits the following description in terms of Dirac
operators:

 \[\slashi \partial_{\C} \mapsto (\slashi \partial_{\R}, \slashi \partial_{\R})\]
This extends the results of \cite{MP} to archimedean fields.
\medskip

We thank Paul Baum for a valuable exchange of emails.

\section{On the tempered dual of $\GL(n)$}

Let $F = \mathbb{R}$.  In order to compute the $K$-theory of the
reduced $C^*$-algebra of $\GL(n,F)$ we need to parametrize the
tempered dual $\mathcal{A}^{t}_{n}(F)$ of $GL(n,F)$.

Let $M$ be a standard Levi subgroup of $\GL(n,F)$, i.e. a
block-diagonal subgroup. Let ${}^{0}M$ be the subgroup of $M$ such
that the determinant of each block-diagonal is $\pm 1$. Denote by
$X(M)=\widehat{M/{}^{0}M}$ the group of \textit{unramified
characters} of $M$, consisting of those characters which are
trivial on ${}^{0}M$.

Let $W(M)=N(M)/M$ denote the Weyl group of $M$. $W(M)$ acts on the
discrete series $E_{2}({}^{0}M)$ of ${}^{0}M$ by permutations.

Now, choose one element $\sigma\in{E_{2}({}^{0}M)}$ for each
$W(M)$-orbit. The \emph{isotropy subgroup} of $W(M)$ is defined to
be
$$W_{\sigma}(M)=\{\omega\in W(M): \omega .\sigma = \sigma\}.$$
Form the disjoint union
\begin{equation}\label{disjoint union}
\bigsqcup_{(M,\sigma)}X(M)/W_{\sigma}(M)=\bigsqcup_{M}\bigsqcup_{\sigma\in{E_{2}(^{0}M)}}X(M)/W_{\sigma}(M).
\end{equation}
The disjoint union has the structure of a locally compact, Hausdorff
space and is called the \textit{Harish-Chandra parameter space}. The
parametrization of the tempered dual
$\mathcal{A}^{t}_{n}(\mathbb{R})$ is due to Harish-Chandra, see
\cite{M}.

\begin{proposition} There exists a bijection
\begin{displaymath}
\begin{array}{cccc}
\bigsqcup_{(M,\sigma)}X(M)/W_{\sigma}(M)&\longrightarrow{}&\mathcal{A}^{t}_{n}(\mathbb{R})\\
\chi{}^{\sigma}&\mapsto & i_{GL(n),MN}(\chi{}^{\sigma}\otimes 1),\\
\end{array}
\end{displaymath}
where $\chi{}^{\sigma}(x):=\chi{(x)}\sigma{(x)}$ for all $x\in{M}$.
\end{proposition}

In view of the above bijection, we will denote the Harish-Chandra
parameter space by $\mathcal{A}^{t}_{n}(\mathbb{R})$.

We will see now the particular features of the archimedean case,
starting with $GL(n,\mathbb{R})$. Since the discrete series of
$GL(n,\mathbb{R})$ is empty for $n\geq{3}$, we only need to consider
partitions of $n$ into 1's and 2's.
This allows us to to decompose $n$ as $n=2q+r$, where $q$ is the
number of 2's and $r$ is the number of 1's in the partition. To this
decomposition we associate the partition
$$n=(\underbrace{2,...,2}_{q},\underbrace{1,...,1}_{r}),$$
which corresponds to the Levi subgroup
$$M\cong\underbrace{GL(2,\mathbb{R})\times{...}\times{}GL(2,\mathbb{R})}_{q}\times\underbrace{GL(1,\mathbb{R})\times{...}\times{}GL(1,\mathbb{R})}_{r}.$$

Varying $q$ and $r$ we determine  a representative in each
equivalence class of Levi subgroups. The subgroup $^{0}M$ of $M$ is
given by
$${}^{0}M\cong\underbrace{SL^{\pm}(2,\mathbb{R})\times{...}\times{}SL^{\pm}(2,\mathbb{R})}_{q}\times\underbrace{SL^{\pm}(1,\mathbb{R})\times{...}\times{}SL^{\pm}(1,\mathbb{R})}_{r},$$
where
$$SL^{\pm}(m,\mathbb{R})=\{g\in{GL(m,\mathbb{R})}:|det(g)|=1\}$$
is the \textit{unimodular subgroup} of $GL(m,\mathbb{R})$. In
particular,
$SL^{\pm}(1,\mathbb{R})=\{\pm{1}\}\cong\mathbb{Z}/2\mathbb{Z}$.

The representations in the discrete series of $GL(2,\mathbb{R})$,
denoted $\mathcal{D}_{\ell}$ for $\ell\in\mathbb{N}$ $(\ell\geq 1)$
are induced from $SL(2,\mathbb{R})$ \cite[p.399]{K}:
\[
\mathcal{D}_{\ell}=ind_{SL^{\pm}(2,\mathbb{R}),SL(2,\mathbb{R})}(\mathcal{D}^{\pm}_{\ell}),
\]
where $\mathcal{D}^{\pm}_{\ell}$ acts in the space
\[
\{f:\mathcal{H}\rightarrow\mathbb{C}|f \textrm{ analytic
},\|f\|^2=\int\int|f(z)|^2y^{\ell-1})dxdy<\infty\}.
\]
Here, $\mathcal{H}$ denotes the Poincar\'{e} upper half plane. The
action of $g= \left( \begin{array}{cc}
 a & b  \\
 c  & d
\end{array} \right)$ is given by
\[
\mathcal{D}^{\pm}_{\ell}(g)(f(z))=(bz+d)^{-(\ell+1)}f(\frac{az+c}{bz+d}).
\]

More generally, an element $\sigma$ from the discrete series
$E_{2}({}^{0}M)$ is given by
\begin{equation}
\sigma{=}i_{G,MN}(\mathcal{D}^{\pm}_{\ell_1}\otimes{...}\otimes{\mathcal{D}^{\pm}_{\ell_q}}\otimes{\tau_{1}}\otimes{...}\otimes\tau_{r}\otimes
1),
\end{equation}
where $\mathcal{D}^{\pm}_{\ell_i}$ $(\ell_i\geq 1)$ are the
discrete series representations of $SL^{\pm}(2,\mathbb{R})$ and
$\tau_{j}$ is a representation of
$SL^{\pm}(1,\mathbb{R})\cong\mathbb{Z}/2\mathbb{Z}$, i.e.
$id=(x\mapsto x)$ or $sgn=(x\mapsto\frac{x}{|x|})$.

Finally we will compute the unramified characters $X(M)$, where $M$
is the Levi subgroup associated to the partition $n=2q+r$.

Let $x\in{GL(2,\mathbb{R})}$. Any character of $GL(2,\mathbb{R})$ is
given by
$$\chi{(det(x))}=(sgn(det(x)))^{\varepsilon}|det(x)|^{it}$$
$(\varepsilon=0,1, t\in\mathbb{R})$ and it is unramified provided
that
$$\chi{(det(g))}=\chi{(\pm{1})}=(\pm{1})^{\varepsilon}=1, \textrm{ for all }g\in{SL^{\pm}}(2,\mathbb{R}).$$
This implies $\varepsilon = 0$ and any unramified character of
$GL(2,\mathbb{R})$ has the form
\begin{equation}\label{unramified character GL(2)}
\chi{(x)}=|det(x)|^{it},\textrm{ for some }t\in\mathbb{R}.
\end{equation}
Similarly, any unramified character of
$GL(1,\mathbb{R})=\mathbb{\mathbb{R}}^{\times}$ has the form
\begin{equation}\label{unramified character GL(1)}
\xi{(x)}=|x|^{it},\textrm{ for some }t\in\mathbb{R}.
\end{equation}
Given a block diagonal matrix $diag(g_{1},...,g_{q},\omega_{1},...,
\omega_{r})\in M$, where $g_{i}\in{GL(2,\mathbb{R})}$ and
$\omega_{j}\in{GL(1,\mathbb{R})}$, we conclude from (\ref{unramified
character GL(2)}) and (\ref{unramified character GL(1)}) that any
unramified character $\chi\in{X(M)}$ is given by
$$\hskip -5.0cm \chi(diag(g_{1},...,g_{q},\omega_{1},..., \omega_{r}))=$$
$$=|det(g_{1})|^{it_{1}}\times{}...\times{}|det(g_{q})|^{it_{q}}\times{}|\omega_{1}|^{it_{q+1}}\times{}...\times{}|\omega_{r}|^{it_{q+r}},$$
for some $(t_{1},...,t_{q+r})\in\mathbb{R}^{q+r}$. We can denote
such element $\chi\in X(M)$ by $\chi_{(t_{1},...,t_{q+r})}$. We have
the following result.

\begin{proposition}\label{X(M),R} Let $M$ be a Levi subgroup of $GL(n,\mathbb{R})$, associated to the partition
$n=2q+r$.
Then, there is a bijection
$$X(M) \rightarrow \mathbb{R}^{q+r} \textrm{ ,
}\chi_{(t_{1},...,t_{q+r})}\mapsto (t_{1},...,t_{q+r}).$$
\end{proposition}

Let us consider now $\GL(n,\mathbb{C})$. The tempered dual of
$\GL(n,\mathbb{C})$ comprises the \emph{unitary principal series}
in accordance with Harish-Chandra \cite[p. 277]{HC}. The
corresponding Levi subgroup is a maximal torus $T
\cong(\mathbb{C}^{\times})^{n}$. It follows that $^{0}T
\cong\mathbb{T}^{n}$ the compact $n$-torus.

The principal series representations are given by
\begin{equation}
\pi_{\ell,it}=i_{G,TU}(\sigma\otimes 1),
\end{equation}
where $\sigma{=}\sigma_{1}\otimes{...}\otimes\sigma_{n}$ and
$\sigma_{j}(z)=(\frac{z}{|z|})^{\ell_{j}}|z|^{it_{j}}$
($\ell_{j}\in\mathbb{Z}$ and $t_{j}\in\mathbb{R}$).

An unramified character is given by
\begin{displaymath}
\chi \left( \begin{array}{cccc}
 z_{1} &  &   \\
   & \ddots &   \\
   &   &  z_{n}
\end{array} \right)
=|z_{1}|^{it_{1}}\times{}...\times{}|z_{n}|^{it_{n}}
\end{displaymath}
and we can represent $\chi$ as $\chi_{(t_{1},...,t_{n})}$.
Therefore, we have the following result.

\begin{proposition}\label{X(M),C}Denote by $T$ the standard
maximal torus in $GL(n,\mathbb{C})$. There is a bijection
$$X(T) \rightarrow \mathbb{R}^{n} \textrm{ ,
}\chi_{(t_{1},...,t_{n})}\mapsto (t_{1},...,t_{n}).$$
\end{proposition}


\section{$K$-theory for $GL(n)$}

Using the Harish-Chandra parametrization of the tempered dual of
$\GL(n,\R)$ and $\GL(n,\C)$ (recall that the Harish-Chandra parameter space is a
locally compact, Hausdorff topological space) we can compute the
$K$-theory of the reduced $C^{*}$-algebras $C_{r}^{*}\GL(n,\R)$ and $C_{r}^{*}\GL(n,\C)$.

\subsection{$K$-theory for $\GL(n,\R)$}

We exploit the strong Morita equivalence described in \cite[\S 1.2
]{Pl1}. We infer that
\begin{equation}\label{K-theory GL(n,R)}
\begin{matrix}
 K_{j}(C_{r}^{*}\GL(n,\mathbb{R}))  &  =  &  K^{j}(\bigsqcup_{(M,\sigma)}X(M)/W_{\sigma}(M))  \cr
 & =  & \bigoplus_{(M,\sigma)}K^{j}(X(M)/W_{\sigma}(M)) \cr
 & =  & \hskip -0.2cm\bigoplus_{(M,\sigma)}K^{j}(\mathbb{R}^{n_{M}}/W_{\sigma}(M)), \cr
\end{matrix}
\end{equation}
where $n_{M}=q+r$ if $M$ is a representative of the equivalence
class of Levi subgroup associated to the partition $n=2q+r$. Hence
the $K$-theory depends on $n$ and on each Levi subgroup.

To compute (\ref{K-theory GL(n,R)}) we have to consider the following orbit spaces:
\begin{itemize}
\item[(i)] $\mathbb{R}^n$, in which case $W_{\sigma}(M)$ is the trivial subgroup of the Weil group $W(M)$;
\item[(ii)] $\mathbb{R}^n/S_n$, where $W_{\sigma}(M)=W(M)$ (this is one of the possibilities for the partition of $n$ into 1's);
\item[(iii)] $\mathbb{R}^n/(S_{n_1}\times...\times S_{n_k})$, where $W_{\sigma}(M)=S_{n_1}\times...\times S_{n_k}\subset W(M)$ (see the examples below).
\end{itemize}

\begin{defn}\label{closed cone}
An orbit space as indicated in $(ii)$ and $(iii)$ is called a closed cone.
\end{defn}

The $K$-theory for $\mathbb{R}^n$ may be summarized as follows

\begin{displaymath}
K^j(\mathbb{R}^n)=
\left\{ \begin{array}{ll}
 \mathbb{Z} \mbox{ if } n=j\mod 2\\
  0 \mbox{ otherwise }.
 \end{array} \right.
\end{displaymath}

The next results show that the $K$-theory of a closed cone vanishes.

\begin{lem}\label{lemma1}
$K^j(\mathbb{R}^n/S_n)=0, j=0,1.$
\end{lem}
\begin{proof}
We need the following definition. A point $(a_1,...,a_n)\in\mathbb{R}^n$ is called normalized if
$a_j\leq a_{j+1}$, for $j=1,2,...,n-1$. Therefore, in each orbit there is exactly one normalized point and
$\mathbb{R}^n/S_n$ is homeomorphic to the subset of $\mathbb{R}^n$ consisting of all normalized points of $\mathbb{R}^n$. We denote the set of all normalized points of $\mathbb{R}^n$ by $N(\mathbb{R}^n)$.

In the case of $n=2$, let $(a_1,a_2)$ be a normalized point of $\mathbb{R}^2$. Then, there is a unique $t\in[1,+\infty[$ such that $a_2=ta_1$ and the map
\[
\mathbb{R}\times[1,+\infty[\rightarrow N(\mathbb{R}^2) , (a,t)\mapsto (a,ta)
\]
is a homeomorphism.

If $n>2$ then the map
\[
N(\mathbb{R}^{n-1})\times[1,+\infty[\rightarrow N(\mathbb{R}^n) , (a_1,...,a_{n-1},t)\mapsto (a_1,...,a_{n-1},ta_n)
\]
is a homeomorphism. Since $[1,+\infty[$ kills both the $K$-theory groups $K^0$ and $K^1$, the result follows by applying K\"{u}nneth formula.
\end{proof}

The symmetric group $S_n$ acts on $\mathbb{R}^n$ by permuting the components. This induces an action of any subgroup $S_{n_1}\times...\times S_{n_k}$ of $S_n$ on $\mathbb{R}^n$. Write
\[
\mathbb{R}^n\cong\mathbb{R}^{n_1}\times\mathbb{R}^{n_2}\times...\times\mathbb{R}^{n_k}\times\mathbb{R}^{n-n_1-...-n_k}.
\]
If $n=n_1+...+n_k$ then we simply have $\mathbb{R}^n\cong\mathbb{R}^{n_1}\times...\times\mathbb{R}^{n_k}$.

The group $S_{n_1}\times...\times S_{n_k}$ acts on $\mathbb{R}^n$ as follows.\\
$S_{n_1}$ permutes the components of $\mathbb{R}^{n_1}$ leaving the remaining fixed;\\
$S_{n_2}$ permutes the components of $\mathbb{R}^{n_2}$ leaving the remaining fixed;\\
and so on. If $n>n_1+...+n_k$ the components of $\mathbb{R}^{n-n_1-...-n_k}$ remain fixed. This can be interpreted, of course, as the action of the trivial subgroup. As a consequence, one identify the orbit spaces
\[
\mathbb{R}^n/(S_{n_1}\times...\times S_{n_k})\cong\mathbb{R}^{n_1}/S_{n_1}\times...\times\mathbb{R}^{n_k}/S_{n_k}\times\mathbb{R}^{n-n_1-...-n_k}
\]

\begin{lem}
$K^j(\mathbb{R}^n/(S_{n_1}\times...\times S_{n_k})=0, j=0,1,$ where $S_{n_1}\times...\times S_{n_k}\subset S_n$.
\end{lem}
\begin{proof}
It suffices to prove for $\mathbb{R}^n/(S_{n_1}\times S_{n_2})$. The general case follows by induction on $k$.

Now, $\mathbb{R}^n/(S_{n_1}\times S_{n_2})\cong\mathbb{R}^{n_1}/S_{n_1}\times \mathbb{R}^{n-n_1}/S_{n_2}$. Applying the K\"{u}nneth formula and Lemma \ref{lemma1}, the result follows.
\end{proof}

We give now some examples by computing
$K_{j}C^{*}_{r}GL(n,\mathbb{R})$ for small $n$.

\begin{example}\label{K-theory GL(1,R)}
We start with the case of $GL(1,\mathbb{R})$. We have:
$$M=\RR \textrm{ , } ^{0}M=\mathbb{Z}/2\mathbb{Z} \textrm{ , }W(M)=1 \textrm{ and
}X(M)=\mathbb{R}.$$ Hence,
\begin{equation}
\mathcal{A}_{1}^{t}(\mathbb{R})\cong\bigsqcup_{\sigma\in\widehat{(\mathbb{Z}/2\mathbb{Z})}}\mathbb{R}/
1=\mathbb{R}\sqcup\mathbb{R},
\end{equation}
and the $K$-theory is given by
\begin{displaymath}
K_{j}C^{*}_{r}GL(1,\mathbb{R})\cong{K^{j}}(\mathcal{A}_{1}^{t}(\mathbb{R}))=K^{j}(\mathbb{R}\sqcup\mathbb{R})=K^{j}(\mathbb{R})\oplus{K^{j}}(\mathbb{R})=
\left\{ \begin{array}{ll}
 \mathbb{Z}\oplus\mathbb{Z} & ,j=1\\
 \hskip 0.5cm 0  &  ,j=0.
 \end{array} \right.
\end{displaymath}
\end{example}

\begin{example}\label{K-theory GL(2,R)} For $GL(2,\mathbb{R})$ we have two partitions of $n=2$ and the following data

\bigskip

\begin{tabular}{|c|c|c|c|c|c|}
\hline
\textbf{Partition} & $M$ & $^{0}M$ & $W(M)$ & $X(M)$ & $\sigma\in{E_{2}(^{0}M)}$\\
\hline
2+0 & $GL(2,\mathbb{R})$ & $SL^{\pm}(2,\mathbb{R})$ & $1$ & $\mathbb{R}$ &
$\sigma=i_{G,P}(\mathcal{D}^{+}_{\ell}),\ell\in\mathbb{N}$\\
1+1 & $(\mathbb{R}^\times)^2$ & $(\mathbb{Z}/2\mathbb{Z})^2$ & $\mathbb{Z}/2\mathbb{Z}$ & $\mathbb{R}^2$ &
$\sigma=i_{G,P}(id\otimes{sgn})$\\
\hline
\end{tabular}

\bigskip

Then the tempered dual is parameterized as follows
$$\mathcal{A}_{2}^{t}(\mathbb{R})\cong\bigsqcup_{(M,\sigma)}X(M)/W_{\sigma}(M)=(\bigsqcup_{\ell\in\mathbb{N}}\mathbb{R})\sqcup(\mathbb{R}^{2}/S_2)\sqcup(\mathbb{R}^{2}/S_2)\sqcup\mathbb{R}^{2},$$
and the $K$-theory groups are given by
\begin{displaymath}
K_{j}C^{*}_{r}GL(2,\mathbb{R})\cong{K^{j}}(\mathcal{A}_{2}^{t}(\mathbb{R}))\cong(\bigoplus_{\ell\in\mathbb{N}}K^{j}(\mathbb{R}))\oplus{K^{j}(\mathbb{R}^{2})}=
\left\{ \begin{array}{ll}
 \bigoplus_{\ell\in\mathbb{N}}\mathbb{Z} & ,j=1\\
 \hskip 0.5cm\mathbb{Z}  &  ,j=0.
 \end{array} \right.
\end{displaymath}
\end{example}

\begin{example} For $GL(3,\mathbb{R})$ there are two partitions for $n=3$, to which correspond the
following data

\bigskip

\begin{tabular}{|c|c|c|c|c|c|}
\hline
\textbf{Partition} & $M$ & $^{0}M$ & $W(M)$ & $X(M)$ \\
\hline
2+1 & $GL(2,\mathbb{R})\times\mathbb{R}^\times$ & $SL^{\pm}(2,\mathbb{R})\times(\mathbb{Z}/2\mathbb{Z})$ & $1$ &
$\mathbb{R}^2$\\
1+1+1 & $(\mathbb{R}^\times)^3$ & $(\mathbb{Z}/2\mathbb{Z})^3$ & $S_{3}$ & $\mathbb{R}^3$\\
\hline
\end{tabular}

\bigskip

For the partition $3=2+1$, an element $\sigma\in{E_{2}(^{0}M)}$ is
given by
$$\sigma=i_{G,P}(\mathcal{D}^{+}_{\ell}\otimes\tau)
\textrm{ , }\ell\in\mathbb{N}\textrm{ and
}\tau\in(\widehat{\mathbb{Z}/2\mathbb{Z}}).$$

\bigskip

For the partition
$3=1+1+1$, an element $\sigma\in{E_{2}(^{0}M)}$ is given by
$$\sigma=i_{G,P}(\bigotimes_{i=1}^{3}\tau_{i})
\textrm{ , }\tau_i\in(\widehat{\mathbb{Z}/2\mathbb{Z}}).$$

The tempered dual is parameterized as follows
$$\mathcal{A}_{3}^{t}(\mathbb{R})\cong\bigsqcup_{(M,\sigma)}X(M)/W_{\sigma}(M)=\bigsqcup_{\mathbb{N}\times(\mathbb{Z}/2\mathbb{Z})}(\mathbb{R}^{2}/
1)\bigsqcup_{(\mathbb{Z}/2\mathbb{Z})^{3}}(\mathbb{R}^{3}/S_{3}).$$

The $K$-theory groups are given by
\begin{displaymath}
K_{j}C^{*}_{r}GL(3,\mathbb{R})\cong{K^{j}}(\mathcal{A}_{3}^{t}(\mathbb{R}))\cong\bigoplus_{\mathbb{N}\times(\mathbb{Z}/2\mathbb{Z})}K^{j}(\mathbb{R}^{2})\oplus{0}=
\left\{ \begin{array}{ll}
 \bigoplus_{\mathbb{N}\times(\mathbb{Z}/2\mathbb{Z})}\mathbb{Z} & ,j=0\\
 \hskip 0.5cm 0  &  ,j=1.
 \end{array} \right.
\end{displaymath}
\end{example}

The general case of $GL(n,\mathbb{R})$ will now be considered. It
can be split in two cases: $n$ even and $n$ odd.

\vspace{11pt}

$\bullet$ $n=2q$ even\\
Suppose $n$ is even. For every partition $n=2q+r$, either
$W_{\sigma}(M)= 1$ or $W_{\sigma}(M)\neq 1$. If $W_{\sigma}(M)\neq
1$ then $\mathbb{R}^{n_{M}}/W_{\sigma}(M)$ is a cone and the
$K$-groups $K^{0}$ and $K^{1}$ both vanish.  This happens precisely
when $r>2$ and therefore we have only two partitions, corresponding
to the choices of $r=0$ and $r=2$, which contribute to the
$K$-theory with non-zero $K$-groups

\bigskip

\begin{tabular}{|c|c|c|c|}
\hline
\textbf{Partition} & $M$ & $^{0}M$ & $W(M)$ \\
\hline
$2q$ & $GL(2,\mathbb{R})^q$ & $SL^{\pm}(2,\mathbb{R})^q$ & $S_{q}$ \\
$2(q-1)+2$ & $GL(2,\mathbb{R})^{q-1}\times(\mathbb{R}^\times)^2$ &
$SL^{\pm}(2,\mathbb{R})^{q-1}\times(\mathbb{Z}/2\mathbb{Z})^2$ & $S_{q-1}\times(\mathbb{Z}/2\mathbb{Z})$ \\
\hline
\end{tabular}

\bigskip

We also have $X(M)\cong\mathbb{R}^q$ for $n=2q$, and
$X(M)\cong\mathbb{R}^{q+1}$, for $n=2(q-1)+2$.

For the partition $n=2q$ $(r=0)$, an element
$\sigma\in{E_{2}(^{0}M)}$ is given by
$$\sigma=i_{G,P}(\mathcal{D}^{+}_{\ell_{1}}\otimes{...}\otimes\mathcal{D}^{+}_{\ell_{q}})
\textrm{ , }(\ell_{1},...\ell_{q})\in\mathbb{N}^{q} \textrm{ and
}\ell_{i}\neq\ell_{j} \textrm{ if }i\neq{j}.$$

For the partition $n=2(q-1)+2$ $(r=2)$, an element
$\sigma\in{E_{2}(^{0}M)}$ is given by
$$\sigma=i_{G,P}(\mathcal{D}^{+}_{\ell_{1}}\otimes{...}\otimes\mathcal{D}^{+}_{\ell_{q-1}}\otimes{id}\otimes{sgn})
\textrm{ , }(\ell_{1},...\ell_{q-1})\in\mathbb{N}^{q-1} \textrm{ and }\ell_{i}\neq\ell_{j} \textrm{ if }i\neq{j}.$$
Therefore, the tempered dual has the following form
$$\mathcal{A}_{n}^{t}(\mathbb{R})=\mathcal{A}_{2q}^{t}(\mathbb{R})=(\bigsqcup_{\ell\in\mathbb{N}^{q}}\mathbb{R}^{q})\sqcup(\bigsqcup_{\ell'\in\mathbb{N}^{q-1}}\mathbb{R}^{q+1})\sqcup \mathcal{C}$$
where $\mathcal{C}$ is a disjoint union of closed cones as in Definition \ref{closed cone}.
\begin{thm}\label{K-theory GL(2q)}
Suppose $n=2q$ is even. Then the $K$-groups are
\begin{displaymath}
K_{j}C^{*}_{r}GL(n,\mathbb{R})\cong \left\{ \begin{array}{ll}
 \bigoplus_{\ell\in\mathbb{N}^{q}}\mathbb{Z} & ,j\equiv{q}(mod2)\\
 \bigoplus_{\ell\in\mathbb{N}^{q-1}}\mathbb{Z} &  ,\textrm{otherwise}.
 \end{array} \right.
\end{displaymath}
If $q=1$ then the direct sum
$\bigoplus_{\ell\in\mathbb{N}^{q-1}}\mathbb{Z}$ will denote a
single copy of $\mathbb{Z}$.
\end{thm}

$\bullet$ $n=2q+1$ odd\\
If $n$ is odd only one partition contributes to the $K$-theory of
$GL(n,\mathbb{R})$ with non-zero $K$-groups:\\

\bigskip

\begin{tabular}{|c|c|c|c|c|}
\hline
\textbf{Partition} & $M$ & $^{0}M$ & $W(M)$ & $X(M)$\\
\hline
$2q+1$ & $GL(2,\mathbb{R})^{q+1}\times\mathbb{R}^\times$ & $SL^{\pm}(2,\mathbb{R})^{q}\times(\mathbb{Z}/2\mathbb{Z})$ &
$S_q$ & $\mathbb{R}^{q+1}$\\
\hline
\end{tabular}

\bigskip

An element $\sigma\in{E_{2}(^{0}M)}$ is given by
$$\sigma=i_{G,P}(\mathcal{D}^{+}_{\ell_{1}}\otimes{...}\otimes\mathcal{D}^{+}_{\ell_{q}}\otimes\tau)
\textrm{ ,
}(\ell_{1},...\ell_{q},\tau)\in\mathbb{N}^{q}\times(\mathbb{Z}/2\mathbb{Z})
\textrm{ and }\ell_{i}\neq\ell_{j} \textrm{ if }i\neq{j}.$$

The tempered dual is given by
$$\mathcal{A}_{n}^{t}(\mathbb{R})=\mathcal{A}_{2q+1}^{t}(\mathbb{R})=(\bigsqcup_{\ell\in(\mathbb{N}^{q}\times(\mathbb{Z}/2\mathbb{Z}))}\mathbb{R}^{q+1})\sqcup \mathcal{C}$$
where $\mathcal{C}$ is a disjoint union of closed cones as in Definition \ref{closed cone}.
\begin{thm}\label{K-theory GL(2q+1)}
Suppose $n=2q+1$ is odd. Then the $K$-groups are
\begin{displaymath}
K_{j}C^{*}_{r}GL(n,\mathbb{R})\cong \left\{ \begin{array}{ll}
 \bigoplus_{\ell\in\mathbb{N}^{q}\times(\mathbb{Z}/2\mathbb{Z})}\mathbb{Z} & ,j\equiv{q+1}(mod2)\\
 \hskip 1.5cm 0 &  ,\textrm{otherwise}.
 \end{array} \right.
\end{displaymath}
Here, we use the following convention: if $q=0$ then the direct sum
is
$\bigoplus_{\mathbb{Z}/2\mathbb{Z}}\mathbb{Z}\cong\mathbb{Z}\oplus\mathbb{Z}$.
\end{thm}

\bigskip

We conclude that the $K$-theory of $C^{*}_{r}GL(n,\mathbb{R})$
depends on essentially one parameter $q$ given by the maximum number
of $2's$ in the partitions of $n$ into $1's$ and $2's$. If $n$ is
even then $q=\frac{n}{2}$ and if $n$ is odd then $q=\frac{n-1}{2}$.

\subsection{$K$-theory for $\GL(n,\C)$}

Let $T^{\circ}$ be the maximal compact subgroup of the maximal
compact torus $T$ of $\GL(n,\C)$. Let $\sigma$ be a unitary
character of $T^{\circ}$. We note that $W = W(T)$, \;$W_{\sigma} =
W_{\sigma}(T)$. If
$W_{\sigma} = 1$ then we say that the orbit $W \cdot \sigma$ is
\emph{generic}.

\begin{thm}\label{K-theory GL(n,C)}
The $K$-theory of $C_{r}^{*}\GL(n,\mathbb{C})$ admits the
following description.  If $n = j \mod 2$ then $K_j$ is free
abelian on countably many generators, one for each generic $W$-
orbit in the unitary dual of $T^{\circ}$, and $K_{j+1} = 0$.
\end{thm}
\begin{proof} We exploit the strong Morita equivalence described
in \cite[Prop. 4.1]{PP}. We have a homeomorphism of locally
compact Hausdorff spaces:
\[\mathcal{A}_{n}^{t}(\C) \cong \bigsqcup X(T)/W_{\sigma}(T)\]
by the Harish-Chandra Plancherel Theorem for complex reductive
groups \cite{HC}, and the identification of the Jacobson topology
on the left-hand-side with the natural topology on the
right-hand-side, as in \cite{PP} . The result now follows from
Lemma 4.3.
\end{proof}

\section{Langlands parameters for $GL(n)$}

The Weil group of $\mathbb{C}$ is simply
\[
W_{\mathbb{C}}\cong \mathbb{C}^{\times},
\]
and the Weil group of $\mathbb{R}$ can be written as disjoint union
\[
W_{\mathbb{R}}\cong \CC \sqcup j\CC,
\]
where $j^2=-1$ and $jcj^{-1}=\overline{c}$ ($\overline{c}$ denotes
complex conjugation). We shall use this disjoint union to describe
the representation theory of $W_{\mathbb{R}}$.

\begin{defn}
An $L$- parameter is a continuous homomorphism
\[
\phi: W_F\rightarrow GL(n,\C)
\]
such that $\phi(w)$ is semisimple for all $w\in W_F$.
\end{defn}

$L$-parameters are also called Langlands parameters. Two $L$-parameters are equivalent if they are cojugate under $GL(n,\C)$.
The set of equivalence classes of $L$-parameters is denoted by $\mathcal{G}_n$. And the set of equivalence classes of $L$-parameters whose image is bounded is denoted by $\mathcal{G}^t_n$.

\medskip

Let $F$ be either $\R$ or $\C$. Let $\mathcal{A}_n(F)$ (resp. $\mathcal{A}^t_n(F)$) denote the smooth dual (resp. tempered dual) of $GL(n,F)$. The local Langlands correspondence is a bijection
$$\mathcal{G}_n(F)\rightarrow\mathcal{A}_n(F).$$
In particular,
$$\mathcal{G}^t_n(F)\rightarrow\mathcal{A}^t_n(F)$$
is also a bijection.

We are only interested in $L$-parameters whose image is bounded. In the sequel we will refer to them, for simplicity, as $L$-parameters.

\bigskip

\textbf{$L$-parameters for $W_{\C}$}

\bigskip

A $1$-dimensional $L$-parameter for $W_{\mathbb{C}}$ is simply a character of $\mathbb{C}^{\times}$ (i.e. a unitary quasicharacter):
\[
\chi(z)=(\frac{z}{|z|})^{\ell}\otimes|z|^{it}
\]
where $|z|=|z|_{\mathbb{C}}=z\overline{z}$, $\ell\in\mathbb{Z}$ and $t\in\mathbb{R}$. To emphasize the dependence on parameters
$(\ell,t)$ we write sometimes $\chi=\chi_{\ell,t}$.

An $n$-dimensional $L$-parameter can be written as a direct sum
of $n$ $1$-dimensional characters of $\mathbb{C}^{\times}$:
\[
\phi=\phi_1\oplus...\oplus\phi_n,
\]
with $\phi_k(z)=(\frac{z}{|z|})^{\ell_k}\otimes|z|^{t_k}, \ell_k\in\mathbb{Z}, t_k\in\mathbb{R}, k=1,...,n.$

\bigskip

\textbf{$L$-parameters for $W_{\R}$}

\bigskip

The $1$-dimensional $L$-parameters for $W_{\mathbb{R}}$ are as follows

\begin{displaymath}
 \left\{ \begin{array}{ll}
 \phi_{\varepsilon,t}(z)=|z|_{\mathbb{R}}^{it}\\
 \\
 \phi_{\varepsilon,t}(j)=(-1)^{\varepsilon}
\end{array} \right.
\textrm{ , }\varepsilon\in\{0,1\}, t\in\mathbb{R}.
\end{displaymath}

We may now describe the local Langlands correspondence for $GL(1,\mathbb{R})$:

\[
\phi_{0,t}\mapsto 1\otimes|.|^{it}_{\mathbb{R}}
\]
\[
\hskip 0.5cm\phi_{1,t}\mapsto sgn\otimes|.|^{it}_{\mathbb{R}}
\]

\medskip

Now, we consider $2$-dimensional $L$-parameters for $W_{\mathbb{R}}$:

\smallskip

\begin{displaymath}
\phi_{\ell,t}(z)=
 \left( \begin{array}{cccc}
 \chi_{\ell,t}(z) & 0  \\
 0 & \overline{\chi}_{\ell,t}(z)
\end{array} \right)
\textrm{ , }\phi_{\ell,t}(j)=
 \left( \begin{array}{cccc}
 0 & (-1)^{\ell}  \\

 1 & 0
\end{array} \right).
\end{displaymath}
with $\ell\in\mathbb{Z}$ and $t\in\mathbb{R}$.

\smallskip

and

\smallskip

\begin{displaymath}
\phi_{m,t,n,s}(z)=
 \left( \begin{array}{cccc}
 \chi_{0,t}(z) & 0  \\

 0 & \chi_{0,s}(z)
\end{array} \right)
\textrm{ , }\phi_{m,t,n,s}(j)=
 \left( \begin{array}{cccc}
(-1)^{m} & 0  \\

 0 & (-1)^{n}
\end{array} \right).
\end{displaymath}
with $m,n\in\{0,1\}$ and $t,s\in\mathbb{R}$.

\bigskip

The local Langlands correspondence for $GL(2,\mathbb{R})$ may be descibed as follows.

\medskip

The $L$-parameter $\phi_{m_1,t_1,m_2,t_2}$ corresponds, via Langlands correspondence, to the unitary principal series:
\[
\phi_{m_1,t_1,m_2,t_2}\mapsto \pi(\mu_1,\mu_2),
\]
where $\mu_i$ is the character of $\mathbb{R}^{\times}$ given by
\[
\mu_i(x)=(\frac{x}{|x|})^{m_i}|x|^{it},  m_i\in\{0,1\}, t_i\in\mathbb{R}.
\]
The $L$-parameter $\phi_{\ell,t}$ corresponds, via the Langlands correspondence, to the discrete series:
\[
\hskip 0.5cm\phi_{\ell,t}\mapsto D_{\ell}\otimes|det(.)|^{it}_{\mathbb{R}} \text{ , } \text{ with } \text{ }\ell\in\mathbb{N}, t\in\mathbb{R}.
\]

\medskip

\begin{proposition}\label{eq. L-parameters}
\begin{itemize}
\item[(i)] $\phi_{\ell,t}\cong\phi_{-\ell,t}$;
\item[(ii)] $\phi_{\ell,t,m,s}\cong\phi_{m,s,\ell,t}$;
\item[(iii)] $\phi_{0,t}\cong\phi_{1,t,0,t}\cong\phi_{0,t,1,t}$;
\end{itemize}
\end{proposition}

The proof is elementary. We now quote the following result.

\begin{lem}\cite{K}\label{irred. rep. W_R}
Every finite-dimensional semi-simple representation $\phi$ of
$W_{\mathbb{R}}$ is fully reducible, and each irreducible
representation has dimension one or two.
\end{lem}

\section{Base change}


We may state the base change problem for archimedean fields in the following way. Consider the archimedean
base change $\mathbb{C}/\mathbb{R}$. We have $W_{\mathbb{C}}\subset{W_\mathbb{R}}$ and there is a natural map
\begin{equation}\label{restriction archimedean}
Res^{W_\mathbb{R}}_{W_\mathbb{C}}:\mathcal{G}_{n}(\mathbb{R})\longrightarrow\mathcal{G}_{n}(\mathbb{C})
\end{equation}
called \textit{restriction}. By the local Langlands correspondence
for archimedean fields \cite[Theorem 3.1, p.236]{BS}\cite{K}, there
is a base change map
\begin{equation}
\mathcal{BC}:\mathcal{A}_{n}(\mathbb{R})\longrightarrow\mathcal{A}_{n}(\mathbb{C})
\end{equation}
such that the following diagram commutes
\[
\xymatrix{\mathcal{A}_{n}(\mathbb{R})\ar[r]^{\mathcal{BC}}
 & \mathcal{A}_{n}(\mathbb{C})  \\
\mathcal{G}_{n}(\mathbb{R})\ar[r]_{Res^{W_\mathbb{R}}_{W_\mathbb{C}}}\ar[u]^{_{\mathbb{R}}\mathcal{L}_{n}}
& \mathcal{G}_{n}(\mathbb{C})\ar[u]_{_{\mathbb{C}}\mathcal{L}_{n}} }
\]

Arthur and Clozel's book \cite{AC} gives a full treatment of base
change for $GL(n)$. The case of archimedean base change can be
captured in an elegant formula \cite[p. 71]{AC}. We briefly review
these results.

Given a partition $n=2q+r$ let $\chi_{i}$ ($i=1,...,q$) be a
ramified character of $\mathbb{C}^{\times}$ and let $\xi_{j}$
($j=1,...,r$) be a ramified character of $\mathbb{R}^{\times}$.
Since the $\chi_{i}$'s are ramified,
$\chi_{i}(z)\neq\chi_{i}^{\tau}(z)=\chi_{i}(\overline{z})$, where $\tau$ is a generator of $Gal(\C/\R)$. By
Langlands classification \cite{K}, each $\chi_{i}$ defines a
discrete series representation $\pi{(\chi_{i})}$ of
$GL(2,\mathbb{R})$, with
$\pi{(\chi_{i})}={\pi{(\chi_{i}^{\tau})}}$. Denote by
$\pi{(\chi_{1},...,\chi_{q},\xi_{1},...,\xi_{r})}$ the
\textit{generalized principal series representation} of
$GL(n,\mathbb{R})$
\begin{equation}\label{generalized principal series rep. GL(n,R)}
\pi{(\chi_{1},...,\chi_{q},\xi_{1},...,\xi_{r})}=i_{GL(n,\mathbb{R}),MN}(\pi{(\chi_{1})}\otimes{...}\otimes\pi{(\chi_{q})}\otimes\xi_{1}\otimes{...}\otimes\xi_{r}\otimes
1).
\end{equation}
The base change map for the general principal series representation
is given by induction from the Borel subgroup $B(\mathbb{C})$
\cite[p. 71]{AC}:
\begin{equation}\label{archimedean base change map}
\mathcal{BC}(\pi)=\Pi{(\chi_{1},...,\chi_{q},\xi_{1},...,\xi_{r})}=i_{GL(n,\mathbb{C}),B(\mathbb{C})}(\chi_{1},\chi_{1}^{\tau},...,\chi_{q},\chi_{q}^{\tau},\xi_{1}\circ{N},...,\xi_{r}\circ{N}),
\end{equation}
where
$N=N_{\mathbb{C}/\mathbb{R}}:\mathbb{C}^{\times}\longrightarrow\mathbb{R}^{\times}$
is the norm map defined by $z\mapsto{z\overline{z}}$.

We illustrate the base change map with two simple examples.
\begin{example}\label{example base change GL(1)}
For $n=1$, base change is simply composition with the norm map
$$\mathcal{BC}:\mathcal{A}_{1}^{t}(\mathbb{R})\rightarrow\mathcal{A}_{1}^{t}(\mathbb{C}) \textrm{ , } \mathcal{BC}(\chi)=\chi\circ{N}.$$
\end{example}

\begin{example}\label{example base change GL(2)}
For $n=2$, there are two different kinds of representations, one for
each partition of $2$. According to (\ref{generalized principal
series rep. GL(n,R)}), $\pi{(\chi)}$ corresponds to the partition
$2=2+0$ and $\pi{(\xi_{1},\xi_{2})}$ corresponds to the partition
$2=1+1$. Then the base change map is given, respectively, by
$$\mathcal{BC}(\pi{(\chi)})=i_{GL(2,\mathbb{C}),B(\mathbb{C})}(\chi,\chi^{\tau}),$$
and
$$\mathcal{BC}(\pi{(\xi_{1},\xi_{2})})=i_{GL(2,\mathbb{C}),B(\mathbb{C})}(\xi_{1}\circ{N},\xi_{2}\circ{N}).$$
\end{example}

\subsection{The base change map}

Now, we define base change as a map of topological spaces and study the induced $K$-theory map.

\begin{proposition}\label{BC proper map}
The base change map
$\mathcal{BC}:\mathcal{A}^{t}_{n}(\mathbb{R})\rightarrow\mathcal{A}^{t}_{n}(\mathbb{C})$
is a continuous proper map.
\end{proposition}
\begin{proof}

First, we consider the case $n=1$. As we have seen in Example
\ref{example base change GL(1)}, base change for $\GL(1)$ is the
map given by $\mathcal{BC}(\chi)=\chi\circ{N}$, for all characters
$\chi\in\mathcal{A}^{t}_{1}(\mathbb{R})$, where
$N:\mathbb{C}^{\times}\rightarrow{\mathbb{R}^{\times}}$ is the
norm map.

Let $z\in\mathbb{C}^{\times}$. We have
\begin{equation}\label{base change of z}
\mathcal{BC}(\chi)(z)=\chi{(|z|^{2})}=|z|^{2it}.
\end{equation}
A generic element from $\mathcal{A}^{t}_{1}(\mathbb{C})$ has the
form
\begin{equation}\label{base change of z generic character}
\mu(z)=(\frac{z}{|z|})^{\ell}|z|^{it},
\end{equation}
where $\ell\in\mathbb{Z}$ and $t\in{S^1}$, as stated before. Viewing
the Pontryagin duals $\mathcal{A}^{t}_{1}(\mathbb{R})$ and
$\mathcal{A}^{t}_{1}(\mathbb{C})$ as topological spaces by
forgetting the group structure, and comparing (\ref{base change of
z}) and (\ref{base change of z generic character}), the base change
map can be defined as the following continuous map
\begin{displaymath}
\begin{array}{cccc}
\varphi : \mathcal{A}^{t}_{1}(\mathbb{R})\cong\mathbb{R}\times(\mathbb{Z}/2\mathbb{Z})&\longrightarrow
&\mathcal{A}^{t}_{1}(\mathbb{C})\cong\mathbb{R}\times\mathbb{Z}\\
\chi=(t,\varepsilon)&\mapsto{}&(2t,0)\\
\end{array}
\end{displaymath}
A compact subset of $\mathbb{R}\times\mathbb{Z}$ in the connected
component $\{\ell\}$ of $\mathbb{Z}$ has the form
$K\times\{\ell\}\subset\mathbb{R}\times\mathbb{Z}$, where
$K\subset\mathbb{R}$ is compact. We have
\begin{displaymath}
\varphi^{-1}(K\times\{\ell\})=
 \left\{ \begin{array}{ll}
\emptyset & ,\textrm{if } \ell\neq{0}\\
 \frac{1}{2}K\times\{\varepsilon\} & ,\textrm{if } \ell={0},
\end{array} \right.
\end{displaymath}
where $ \varepsilon\in\mathbb{Z}/2\mathbb{Z}$. Therefore
$\varphi^{-1}(K\times\{\ell\})$ is compact and $\varphi$ is proper.

\vspace{11pt}

The Case $n>1$. Base change determines a map
$\mathcal{BC}:\mathcal{A}^{t}_{n}(\mathbb{R})\rightarrow\mathcal{A}^{t}_{n}(\mathbb{C})$
of topological spaces. Let $X=X(M)/W_{\sigma}(M)$ be a connected
component of $\mathcal{A}_n^t(\mathbb{R})$. Then, $X$ is mapped
under $BC$ into a connected component $Y=Y(T)/W_{\sigma'}(T)$ of
$\mathcal{A}_n^t(\mathbb{C})$. Given a generalized principal
series representation
$$\pi(\chi_1,...,\chi_q,\xi_1,...,\xi_r)$$
where the $\chi_i$'s are ramified characters of
$\mathbb{C}^{\times}$ and the $\xi$'s are ramified characters of
$\mathbb{R}^{\times}$, then
$$\mathcal{BC}(\pi)=i_{G,B}(\chi_1,\chi_1^{\tau},...,\chi_q,\chi_q^{\tau},\xi_1\circ N,...,\xi_r\circ N).$$
Here, $N=N_{\mathbb{C}/\mathbb{R}}$ is the norm map and $\tau$ is the generator of $Gal(\mathbb{C}/\mathbb{R})$.

We associate to $\pi$ the usual parameters uniquely defined for each character $\chi$ and $\xi$. For simplicity, we write the set of parameters as a $(q+r)$-uple:
\[
(t,t')=(t_1,...,t_q,t'_1,...,t'_r)\in\mathbb{R}^{q+r}\cong X(M).
\]

Now, if $\pi(\chi_1,...,\chi_q,\xi_1,...,\xi_r)$ lies in the connected component defined by the fixed parameters $(\ell,\varepsilon)\in\mathbb{Z}^q\times(\mathbb{Z}/2\mathbb{Z})^r$, then
\[
(t,t')\in X(M)\mapsto(t,t,2t')\in Y(T)
\]
is a continuous proper map.

It follows that
\[
\mathcal{BC}: X(M)/W_{\sigma}(M) \rightarrow Y(T)/W_{\sigma'}(T)
\]
is continuous and proper since the orbit spaces are endowed with the quotient topology.
\end{proof}

\begin{thm}\label{main result archimedean base change n>1}
The functorial map induced by base change
$$K_{j}(C_{r}^{*}GL(n,\mathbb{C}))\stackrel{K_{j}(\mathcal{BC})}\longrightarrow{K_{j}(C_{r}^{*}GL(n,\mathbb{R}))}$$
is zero for $n>1$.
\end{thm}
\begin{proof} We start with the case $n>2$. Let $n=2q+r$ be a partition and $M$
the associated Levi subgroup of $GL(n,\mathbb{R})$. Denote by
$X_{\mathbb{R}}(M)$ the unramified characters of $M$. As we have
seen, $X_{\mathbb{R}}(M)$ is parametrized by $\mathbb{R}^{q+r}$. On
the other hand, the only Levi subgroup of $GL(n,\mathbb{C})$ for
$n=2q+r$ is the diagonal subgroup
$X_{\mathbb{C}}(M)=(\mathbb{C}^{\times})^{2q+r}$.

If $q=0$ then $r=n$ and both $X_{\mathbb{R}}(M)$ and
$X_{\mathbb{C}}(M)$ are parametrized by $\mathbb{R}^{n}$. But then
in the real case an element $\sigma\in{E_{2}({}^{0}M)}$ is given by
$$\sigma{=}i_{GL(n,\mathbb{R}),P}(\chi_{1}\otimes{...}\otimes\chi_{n}),$$
with $\chi_{i}\in\widehat{\mathbb{Z}/2\mathbb{Z}}$. Since
$n\geq{3}$ there is always repetition of the $\chi_{i}$'s. It
follows that the isotropy subgroups $W_{\sigma}(M)$ are all
nontrivial and the quotient spaces $\mathbb{R}^{n}/W_{\sigma}$ are
closed cones. Therefore, the $K$-theory groups vanish.

If ${q}\neq{0}$, then $X_{\mathbb{R}}(M)$ is parametrized by
$\mathbb{R}^{q+r}$ and $X_{\mathbb{C}}(M)$ is parametrized by
$\mathbb{R}^{2q+r}$ (see Propositions \ref{X(M),R} and
\ref{X(M),C}).

Base change creates a map
$$\mathbb{R}^{q+r}\longrightarrow\mathbb{R}^{2q+r}.$$
Composing with the stereographic projections we obtain a map
$$S^{q+r}\longrightarrow{S}^{2q+r}$$
between spheres. Any such map is nullhomotopic  \cite[Proposition
17.9]{BT}. Therefore, the induced $K$-theory map
$$K^{j}(S^{2q+r})\longrightarrow{K^{j}({S}^{q+r})}$$
is the zero map.

\vspace{11pt}

The Case $n=2$. For $n=2$ there are two Levi subgroups of
$GL(2,\mathbb{R})$, $M_{1}\cong GL(2,\mathbb{R})$ and the diagonal
subgroup $M_{2}\cong (\mathbb{R}^{\times})^{2}$. By Proposition
\ref{X(M),R} $X(M_{1})$ is parametrized by $\mathbb{R}$ and
$X(M_{2})$ is parametrized by $\mathbb{R}^{2}$. The group
$GL(2,\mathbb{C})$ has only one Levi subgroup, the diagonal
subgroup $M\cong({\mathbb{C}^{\times}})^{2}$. From Proposition
\ref{X(M),C} it is parametrized by $\mathbb{R}^{2}$.

Since $K^{1}(\mathcal{A}^{t}_{2}(\mathbb{C}))=0$ by Theorem 5.1,
we only have to consider the $K^0$ functor. The only contribution
to $K^{0}(\mathcal{A}^{t}_{2}(\mathbb{R}))$ comes from
$M_{2}\cong{(\mathbb{R}^{\times})^{2}}$ and we have (see Example
\ref{K-theory GL(2,R)})
$$K^{0}(\mathcal{A}^{t}_{2}(\mathbb{R}))\cong\mathbb{Z}.$$
For the Levi subgroup $M_{2}\cong(\mathbb{R}^{\times})^{2}$, base
change is
\begin{displaymath}
\begin{array}{cccc}
\mathcal{BC}:\mathcal{A}^{t}_{2}(\mathbb{R}) & \longrightarrow{} & \mathcal{A}^{t}_{2}(\mathbb{C})\\
\pi{(\xi_{1},\xi_{2})}&\mapsto{}&i_{GL(2,\mathbb{C}),B(\mathbb{C})}(\xi_{1}\circ{N},\xi_{2}\circ{N}),\\
\end{array}
\end{displaymath}
Therefore, it maps a class $[t_1,t_2]$, which lies in the connected component $(\varepsilon_1,\varepsilon_2)$, into the class $[2t_1,2t_2]$, which lies in the connect component $(0,0)$. In other words, base change maps a generalized principal series $\pi(\xi_1,\xi_2)$ into a nongeneric point of $\mathcal{A}_2^t(\mathbb{C})$. It follows from Theorem \ref{K-theory GL(n,C)} that
$$K^0(\mathcal{BC}):K^0(\mathcal{A}_2^t(\R))\rightarrow K^0(\mathcal{A}_2^t(\mathbb{C}))$$
is the zero map.
\end{proof}

\subsection{Base change in one dimension}

In this section we consider base change for $\GL(1)$.

\begin{thm}\label{main result archimedean base change n=1}
The functorial map induced by base change
$$K_{1}(C_{r}^{*}\GL(1,\mathbb{C}))\stackrel{K_{1}(\mathcal{BC})}\longrightarrow{K_{1}(C_{r}^{*}\GL(1,\mathbb{R}))}$$
is given by $K_{1}(\mathcal{BC})=\Delta\circ Pr$, where $Pr$ is the projection of the zero component of $K^1(\mathcal{A}_1^t(\mathbb{C}))$ into $\mathbb{Z}$ and $\Delta$ is the diagonal $\mathbb{Z}\rightarrow\mathbb{Z}\oplus\mathbb{Z}$.
\end{thm}
\begin{proof}
For $\GL(1)$, base change
$$\chi\in\mathcal{A}^t_1(\mathbb{R})\mapsto\chi\circ N_{\mathbb{C}/\mathbb{R}}\in\mathcal{A}^{t}_{1}(\mathbb{C})$$
induces a map
$$K^{1}(\mathcal{BC}):K^{1}(\mathcal{A}^{t}_{1}(\mathbb{C}))\rightarrow{}K^{1}(\mathcal{A}^{t}_{1}(\mathbb{R})).$$
Any character $\chi\in\mathcal{A}^{t}_{1}(\mathbb{R})$ is uniquely determined by a pair of parameters $(t,\varepsilon)\in\mathbb{R}\times\mathbb{Z}/2\mathbb{Z}$. Similarly, any character $\mu\in\mathcal{A}^{t}_{1}(\mathbb{C})$ is uniquely determined
by a pair of parameters $(t,\ell)\in\mathbb{R}\times\mathbb{Z}$. The discrete parameter $\varepsilon$ (resp., $\ell$) labels each connected component of $\mathcal{A}^{t}_{1}(\mathbb{R})=\mathbb{R}\sqcup\mathbb{R}$ (resp., $\mathcal{A}^{t}_{1}(\mathbb{C})=\bigsqcup_{\mathbb{Z}}\mathbb{R}$).

Base change maps each component $\varepsilon$ of
$\mathcal{A}^{t}_{1}(\mathbb{R})$ into the component $0$ of
$\mathcal{A}^{t}_{1}(\mathbb{C})$, sending $t\in\mathbb{R}$ to
$2t\in\mathbb{R}$. The map $t \mapsto 2t$ is homotopic to the
identity. At the level of $K^1$, the base change map is given by
$K_{1}(BC)=\Delta\circ Pr$, where $Pr$ is the projection of the
zero component of $K^1(\mathcal{A}_1^t(\mathbb{C}))$ into
$\mathbb{Z}$ and $\Delta$ is the diagonal
$\mathbb{Z}\rightarrow\mathbb{Z}\oplus\mathbb{Z}$.

\end{proof}

\section{Automorphic induction}

We begin this section by describing the action of $Gal(\mathbb{C}/\mathbb{R})$ on $\widehat{W_{\mathbb{C}}}=\widehat{\mathbb{C}^{\times}}$.
Take $\chi=\chi_{\ell,t}\in\widehat{\mathbb{C}^{\times}}$ and let $\tau$ denote the nontrivial element of $Gal(\mathbb{C}/\mathbb{R})$. Then,
$Gal(\mathbb{C}/\mathbb{R})$ acts on $\widehat{\mathbb{C}^{\times}}$ as follows:
\[
\chi^{\tau}(z)=\chi(\overline{z}).
\]
Hence,
\[
\chi^{\tau}_{\ell,t}(z)=(\frac{\overline{z}}{|z|})^{\ell}|z|^{it}_{\mathbb{C}}=(\frac{z}{|z|})^{-\ell}|z|^{it}_{\mathbb{C}}
\]
and we conclude that
\[
\chi^{\tau}_{\ell,t}(z)=\chi_{-\ell,t}(z).
\]

In particular,

\[
\chi^{\tau}=\chi\Leftrightarrow \ell=0\Leftrightarrow \chi=|.|^{it}_{\mathbb{C}}
\]

i.e, $\chi$ is unramified.

\bigskip

Note that $W_{\C}\subset W_{\R}$, with index $[W_{\R}:W_{\C}]=2$. Therefore, there is a natural induction map
\[
Ind_{\C/\R}:\mathcal{G}_1^t(\C)\rightarrow\mathcal{G}_{2}^t(\R).
\]
By the local Langlands correspondence for archimedean fields
\cite{BS,K}, there exists an automorphic induction map $\mathcal{AI}_{\C/\R}$ such that the following diagram commutes
\[
\xymatrix{\mathcal{A}_{1}^t(\mathbb{C})\ar[r]^{\mathcal{AI}_{\C/\R}}
 & \mathcal{A}_{2}^t(\mathbb{R})  \\
\mathcal{G}_{1}^t(\mathbb{C})\ar[r]_{Ind_{\C/\R}}\ar[u]^{_{\mathbb{C}}\mathcal{L}_{1}}
&
\mathcal{G}_{2}^t(\mathbb{R})\ar[u]_{_{\mathbb{R}}\mathcal{L}_{2}}
}
\]

The next result describes reducibility of induced representations.

\begin{proposition}
Let $\chi$ be a character of $W_{\C}$. We have:
\begin{itemize}
\item[(i)] If $\chi\neq\chi^{\tau}$ then $Ind_{\C/\R}(\chi)$ is irreducible;
\item[(ii)] If $\chi=\chi^{\tau}$ then $Ind_{\C/\R}(\chi)$ is reducible. Moreover, there exist $\rho\in\widehat{W_{\mathbb{R}}}$ such that $$Ind_{\C/\R}(\chi)=\rho\oplus\rho^{\tau}=\rho\oplus sgn.\rho,$$
    where $\rho_{|W_{\mathbb{C}}}=\chi$;
\item[(iii)] $Ind_{\C/\R}(\chi_1)\cong Ind_{\C/\R}(\chi_2)$ if and only if $\chi_1=\chi_2$ of $\chi_1=\chi_2^{\tau}$.
\end{itemize}
\end{proposition}

\begin{proof}
Apply Frobenius reciprocity

$$Hom_{W_{\mathbb{R}}}(Ind_{\C/\R}(\chi_1),Ind_{\C/\R}(\chi_2))\cong Hom_{W_{\mathbb{C}}}(\chi_1,\chi_2).$$
Now, $W_{\mathbb{R}}=W_{\mathbb{C}}\sqcup jW_{\mathbb{C}}$. Therefore, the restriction of $Ind_{\C/\R}(\chi)$ to $W_{\mathbb{C}}$ is $\chi\oplus \chi^{\tau}$. The result follow since $Ind_{\C/\R}(\chi)$ is semi-simple.
\end{proof}

\begin{proposition}\label{Prop_irred_W_R}
A finite dimensional continuous irreducible representation of $W_{\R}$ is either a character or isomorphic to some $Ind_{\C/\R}(\chi)$, with $\chi\neq\chi^{\tau}$.
\end{proposition}

\begin{proof}
It follows immediately from Lemma \ref{irred. rep. W_R}
\end{proof}

\subsection{The automorphic induction map}

In this section we describe automorphic induction map as a map of topological spaces. We begin by considering $n=1$.

Let $\chi=\chi_{\ell,t}$ be a character of $W_{\C}$. If $\chi\neq\chi^{\tau}$, by proposition \ref{eq. L-parameters}, \\ $\phi_{\ell,t}\simeq\phi_{-\ell,t}$. Hence,
$$\mathcal{AI}_{\C/\R}(_{\mathbb{C}}\mathcal{L}_{1}(\chi_{\ell,t}))=D_{|\ell|}\otimes|det(.)|^{it}.$$
On the other hand, if $\chi=\chi^{\tau}$ then $\chi=\chi_{0,t}$ and $\chi(z)=|z|^{it}_{\mathbb{C}}$. Therefore,

$$\mathcal{AI}_{\C/\R}(_{\mathbb{C}}\mathcal{L}_{1}(|.|_{\mathbb{C}}^{it}))= {}_{\mathbb{R}}\mathcal{L}_{2}(\rho\oplus sgn.\rho)=\pi(\rho,\rho^{-1}),$$
where $\pi(\rho,\rho^{-1})$ is a reducible principal series and $\rho$ is the character of\\
 $\R^{\times}\simeq W_{\R}^{ab}$ associated with $\chi_{0,t}=|.|_{\mathbb{C}}^{it}$ via class field theory, i.e. $\rho_{|W_{\C}}=\chi$.

\medskip

Recall that
$$\mathcal{A}_1^t(\C)\cong\bigsqcup_{\ell\in\mathbb{Z}}\mathbb{R}$$
and
$$\mathcal{A}_2^t(\R)\cong\Bigg(\bigsqcup_{\ell\in\mathbb{N}}\mathbb{R}\Bigg)\bigsqcup(\mathbb{R}^2/S_2)\bigsqcup(\mathbb{R}^2/S_2)\bigsqcup\mathbb{R}^2.$$

As a map of topological spaces, automorphic induction for $n=1$ may be described as follows:
\begin{equation}
(t,\ell)\in\mathbb{R}\times\mathbb{Z}\mapsto(t,|\ell|)\in\mathbb{R}\times\mathbb{N}, \textrm{ if }\ell\neq 0
\end{equation}
\begin{equation}
(t,0)\in\mathbb{R}\times\mathbb{Z}\mapsto (t,t)\mapsto\mathbb{R}^2, \textrm{ if }\ell=0.
\end{equation}
More generally, let $\chi_1\oplus...\oplus\chi_n$ be an $n$ dimensional $L$-parameter of $W_{\C}$. Then, either $\chi_k\neq\chi_k^{\tau}$ for every $k$, in which case automorphic induction is
\begin{equation}\label{top aut ind for n}
\mathcal{AI}_{\C/\R}(_{\mathbb{C}}\mathcal{L}_{n}(\chi_1\oplus...\oplus\chi_n))=D_{|\ell_1|}\otimes|det(.)|^{it_1}\oplus...\oplus D_{|\ell_n|}\otimes|det(.)|^{it_n}
\end{equation}
or for some $k$ (possibly more than one), $\chi_k=\chi_k^{\tau}$, in which case we have
\begin{equation}
\mathcal{AI}_{\C/\R}(_{\mathbb{C}}\mathcal{L}_{n}(\chi_1\oplus...\oplus|.|_{\mathbb{C}}^{it_k}\oplus...\oplus\chi_n))=
D_{|\ell_1|}\otimes|det(.)|^{it_1}\oplus...
\end{equation}
$$ \oplus\pi(\rho_k,\rho_k^{-1})\oplus...\oplus D_{|\ell_n|}\otimes|det(.)|^{it_n}.$$

In order to describe automorphic induction as a map of topological spaces, it is enough to consider components of $\mathcal{A}_n^t(\C)$ with generic $W$-orbit. For convenience, we introduce the following notation:

if $(t_1,...,t_n)$ is in the component of $\mathcal{A}_n^t(\C)$ labeled by $(\ell_1,...,\ell_n)$, i.e
$$(t_1,...,t_n)\in (\R\times\{\ell_1\})\times ... \times (\R\times\{\ell_n\})$$
we write simply
$$(t_1,...,t_n)\in\R^n_{(\ell_1,...,\ell_n)} \, , \,\,\, \ell_i\in\Z.$$

There are two cases:

\textbf{Case 1:} $\chi_k\neq\chi_k^{\tau}$, i.e. $\ell_k\neq 0$, for every $k$,
\begin{equation}
\mathcal{AI} : \, (t_1,...,t_n)\in\R^n_{(\ell_1,...,\ell_n)}\longmapsto (t_1,...,t_n)\in\R^n_{(|\ell_1|,...,|\ell_n|)}
\end{equation}
So, $(|\ell_1|,...,|\ell_n|)\in\N^n$.

\textbf{Case 2:} if there are $0<m<n$ characters such that $\chi_k=\chi_k^{\tau}$, then
\begin{equation}
\mathcal{AI} : \, (t_1,...,t_k,...,t_n)\in\R^n_{(\ell_1,...,0,...,\ell_n)}\longmapsto (t_1,...,t_k,t_k,...,t_n)\in(\R^n/W)_{(|\ell_1|,...,|\ell_n|)^*}
\end{equation}
where $(|\ell_1|,...,|\ell_n|)^*\in\N^{n-m}$ means that we have deleted the $m$ labels corresponding to $\ell_k=0$. Note that if $m>1$, necessarily $W\neq 0$.

We have the following result

\begin{proposition}
The automorphic induction map
$$\mathcal{AI}_{\C/\R}:\mathcal{A}^{t}_{n}(\mathbb{C})\rightarrow\mathcal{A}^{t}_{2n}(\mathbb{R})$$
is a continuous proper map.
\end{proposition}
The proof follows from the above discussion and is similar to that of proposition \ref{BC proper map}.

\begin{example}
Consider $n=3$. Then,
$$\mathcal{A}_3^t(\C)\simeq \bigsqcup_{\sigma}\R^3/W_{\sigma}$$
and
$$\mathcal{A}_6^t(\R)\simeq (\bigsqcup_{\ell\in\mathbb{N}^3}\R^3) \sqcup (\bigsqcup_{\ell'\in\mathbb{N}^2}\R^4) \sqcup \mathcal{C},$$
where $\mathcal{C}$ is a disjoint union of cones.

\smallskip

Let $\chi_1\oplus\chi_2\oplus\chi_3$ denote a $3$-dimensional $L$-parameter of $W_{\C}$. We have the following description of $\mathcal{AI}_{\C/\R}$ as a map of topological spaces:

\medskip

$\bullet$ $\chi_1\neq\chi_1^{\tau}, \chi_2\neq\chi_2^{\tau}, \chi_3\neq\chi_3^{\tau}$

\smallskip

$$(t_1,t_2,t_3)\in\R^3_{(\ell_1,\ell_2,\ell_3)}\longmapsto (t_1,t_2,t_3)\in\R^3_{(|\ell_1|,|\ell_2|,|\ell_3|)}$$
with $\ell_i\in\Z\backslash\{0\}$.

\medskip

$\bullet$ $\chi_1=\chi_1^{\tau}, \chi_2\neq\chi_2^{\tau}, \chi_3\neq\chi_3^{\tau}$

\smallskip

$$(t_1,t_2,t_3)\in\R^3_{(0,\ell_2,\ell_3)}\longmapsto (t_1,t_1,t_2,t_3)\in(\R^4/W)_{(|\ell_2|,|\ell_3|)}$$
with $\ell_i\in\Z\backslash\{0\}$. Similar for the cases $(\ell_1,0,\ell_3)$ and $(\ell_1,\ell_2,0)$.

\medskip

$\bullet$ $\chi_1=\chi_1^{\tau}, \chi_2=\chi_2^{\tau}, \chi_3\neq\chi_3^{\tau}$

\smallskip

$$(t_1,t_2,t_3)\in\R^3_{(0,0,\ell_3)}\longmapsto (t_1,t_1,t_2,t_2,t_3)\in(\R^5/W)_{(|\ell_3|)}$$
with $\ell_3\in\Z\backslash\{0\}$. Similar for the cases $(\ell_1,0,0)$ and $(0,\ell_2,0)$.

\medskip

$\bullet$ $\chi_1=\chi_1^{\tau}, \chi_2=\chi_2^{\tau}, \chi_3=\chi_3^{\tau}$

\smallskip

$$(t_1,t_2,t_3)\in\R^3_{(0,0,0)}\longmapsto (t_1,t_1,t_2,t_2,t_3,t_3)\in(\R^6/W)$$
\end{example}

\subsection{Automorphic induction in one dimension}

Automorphic induction $\mathcal{AI}$ induces a $K$-theory map at the level of $K$-theory groups $K^1$:

\begin{eqnarray}
K^1(\mathcal{AI}):K^1(\mathcal{A}_2^t(\R))\rightarrow K^1(\mathcal{A}_1^t(\C)).
\end{eqnarray}

We have
$$K^1(\mathcal{A}_2^t(\R))\cong\bigoplus_{\ell\in\mathbb{N}}\mathbb{Z}.$$

Each class of $1$-dimension $L$-parameters of $W_{\C}$ (characters of $\mathbb{C}^{\times}$)
$$[\chi]=[\chi_{\ell,t}]=[\chi_{-\ell,t}] \text{ } (\ell\neq 0)$$
contributes with one generator to $K^1(\mathcal{A}_2^t(\R))$. Note that, under $\mathcal{AI}$, this is precisely the parametrization given by the discrete series $D_{|\ell|}$.

\bigskip

On the other hand,
$$K^1(\mathcal{A}_1^t(\C))\cong\bigoplus_{\ell\in\mathbb{Z}}\mathbb{Z}.$$

Again, each class (of characters of $\mathbb{C}^{\times}$) $[\chi]$ contributes with a generator to $K^1(\mathcal{A}_1^t(\C))$, only this time $[\chi_{\ell,t}]\neq[\chi_{-\ell,t}]$, i.e, $\ell$ and $-\ell$ belong to different classes.

\medskip

Note that we may write

$$K^1(\mathcal{A}_2^t(\R))\cong\bigoplus_{\textrm{ Discrete series }}\mathbb{Z}=\bigoplus_{[D_{|\ell|}]}\mathbb{Z}=\bigoplus_{\ell\in\mathbb{N}}\mathbb{Z}$$

and

$$K^1(\mathcal{A}_1^t(\C))\cong\bigoplus_{\widehat{\mathbb{C}^{\times}}}\mathbb{Z}=\bigoplus_{[\chi_{\ell}]}\mathbb{Z}=\bigoplus_{\ell\in\mathbb{Z}}\mathbb{Z}.$$

The automorphic induction map

\[
K^1(\mathcal{AI}):K^1(\mathcal{A}_2^t(\R))\rightarrow K^1(\mathcal{A}_1^t(\C))
\]

may be interpreted, at the level of $K^1$, as a kind of "shift" map

$$[D_{|\ell|}]\mapsto [\chi_{|\ell|}]$$

More explicitly, the "shift" map is

$$\bigoplus_{\ell\in\mathbb{N}}\mathbb{Z}\rightarrow \bigoplus_{\ell\in\mathbb{Z}}\mathbb{Z}, ([D_1], [D_2],...)\mapsto (...,0,0,[\chi_1], [\chi_2],...)$$
where the image under $K^1(\mathcal{AI})$ on each component of the right hand side with label $\ell\leq 0$ is zero (because $K^1(\mathcal{AI})$ is a group homomorphism so it must map zero into zero).

\subsection{Automorphic induction in $n$ dimensions}

In this section we consider automorphic induction for $GL(n)$. Contrary to base change (see theorems \ref{main result archimedean base change n>1} and \ref{main result archimedean base change n=1}), the $K$-theory map of automorphic induction is nonzero for every $n$.

\begin{thm}\label{main result automorphic induction n}
The functorial map induced by automorphic induction
$$K_{j}(C_{r}^{*}\GL(2n,\mathbb{R}))\stackrel{K_{j}(\mathcal{AI})}\longrightarrow{K_{j}(C_{r}^{*}\GL(n,\mathbb{C}))}$$
is given by
$$[D_{|\ell_1|}\otimes...\otimes D_{|\ell_n|}]\longmapsto [\chi_{|\ell_1|}\oplus...\oplus\chi_{|\ell_n|}]$$
if $n\equiv j\, (mod\, 2)$ and $\chi_k\neq\chi_k^{\tau}$ for every $k$, and is zero otherwise.

\smallskip

Here, $[D_{|\ell_1|}\otimes...\otimes D_{|\ell_n|}]$ denotes the generator of the component $\Z_{(|\ell_1|,...,|\ell_n|)}$ of $K_{j}(C_{r}^{*}\GL(2n,\mathbb{R}))$ and $[\chi_{|\ell_1|}\oplus...\oplus\chi_{|\ell_n|}]$ is the generator of the component $\Z_{(|\ell_1|,...,|\ell_n|)}$ of $K_{j}(C_{r}^{*}\GL(n,\mathbb{C}))$
\end{thm}

\begin{proof}

Let $0\leq m<n$ be the number of characters $\chi_k$ with $\chi_k=\chi_k^{\tau}$.

\smallskip

\textbf{Case $1$:} $m=0$

\smallskip

In this case, $\chi_k\neq\chi_k^{\tau}$ for every $k$. Each character $\chi_{\ell_k}$, $\ell_k\neq 0$, is mapped via the local langlands correspondence into a discrete series $D_{|\ell_k|}$. At the level of $K$-theory, a generator $[D_{|\ell_k|}]$ is mapped into $[\chi_{|\ell_k|}]$. The result follows from (\ref{top aut ind for n}).

\smallskip

\textbf{Case $2$:} $m>0$ odd

\smallskip

Then, if $n\equiv j\, (mod\, 2)$, $K^j(\R^{n+m})=0$ and $K_j(\mathcal{AI})$ is zero.

\smallskip

\textbf{Case $3$:} $m>0$ is even

\smallskip

In this case $K^j(\R^{n})=K^j(\R^{n+m})$. However, $X_{\R}(M)\simeq \R^{n+m}$ corresponds precisely to the partition of $2n$ into $1's$ and $0's$ given by
$$2n=2(n-m)+2m$$
Hence, the number of $1's$ in the partition is $2m\geq 4$. It follows that $(t_1,...,t_n)$ is mapped into a cone and, as a consequence, $K_j(\mathcal{AI})$ is zero.

This concludes the proof.
\end{proof}

\section{Connections with the Baum-Connes correspondence}

The standard maximal compact subgroup of $\GL(1,\C)$ is the circle
group $U(1)$, and the maximal compact subgroup of $\GL(1,\R)$ is
$\Z/2\Z$. Base change for $K^1$ creates a map
\[
\mathcal{R}(\U(1))\longrightarrow\mathcal{R}(\mathbb{Z}/2\mathbb{Z})
\]
where $\mathcal{R}(\U(1))$ is the representation ring of the
circle group $\U(1)$ and $\mathcal{R}(\mathbb{Z}/2\mathbb{Z})$ is
the representation ring of the group $\mathbb{Z}/2\mathbb{Z}$.
This map sends the trivial character of $\U(1)$ to
$1\oplus\varepsilon$, where $\varepsilon$ is the nontrivial
character of $\mathbb{Z}/2\mathbb{Z}$, and sends all the other
characters of $\U(1)$ to zero.

This map has an interpretation in terms of $K$-cycles. The real
line $\R$ is a universal example for the action of $\R^{\times}$
and $\C^{\times}$. The $K$-cycle \[(C_0(\R), L^2(\R),id/dx)\] is
equivariant with respect to $\CC$ and $\RR$, and therefore
determines a class $\slashi
\partial_{\C} \in K_1^{\CC}(\underline{E}\CC)$ and a class
$\slashi \partial_{\R} \in K_1^{\RR}(\underline{E}\RR)$.  On the
left-hand-side of the Baum-Connes correspondence, base change in
dimension $1$ admits the following description in terms of Dirac
operators:

 \[\slashi \partial_{\C} \mapsto (\slashi \partial_{\R}, \slashi \partial_{\R})\]
\medskip

It would be interesting to interpret the automorphic induction map at the level of representation rings:

$$\mathcal{AI}^* : \mathcal{R}(O(2n))\longrightarrow \mathcal{R}(U(n)).$$




S. Mendes, ISCTE, Av. das For\c{c}as Armadas, 1649-026, Lisbon,
Portugal\\ Email: sergio.mendes@iscte.pt\\

R.J. Plymen, Southampton University, Southampton SO17 1BJ,  England
\emph{and} School of Mathematics, Manchester University, Manchester M13 9PL, England
\\ Email: r.j.plymen@soton.ac.uk \quad plymen@manchester.ac.uk\\


\begin{thebibliography}{99}

\bibitem{AC}
J. Arthur, L. Clozel, Simple algebras, base change, and the advanced
theory of the trace formula, Annals of Math. Studies 120, Princeton
University Press, Princeton, 1989.

\bibitem{BCH}
P.~Baum, A.~Connes, N.~Higson, Classifying space for proper actions
and $K$-theory for group $C^*$-algebras, Contemporary Math., 167
(1994) 241--291.


\bibitem{BS} J. Bernstein, S. Gelbart et al, An introduction to the Langlands program, Birkhauser 2003.

\bibitem{BT}
R. Bott and L.W. Tu, Differential Forms in Algebraic Topology,
Springer-Verlag, New York 1982.

\bibitem{CEN} J. Chabert, S. Echterhoff, R. Nest, The
Connes-Kasparov conjecture for almost connected groups and for
linear $p$-adic groups, Publications Math. I.H.E.S. 97 (2003)
239-278.

\bibitem{HC}
Harish-Chandra, Collected papers, Vol. 4, Springer, Berlin (1984).

\bibitem{He}
G. Henniart, {Induction automorphe pour
$\GL(n,\mathbb{C})$}, J. Functional Analysis 258, issue 9 (2010) 3082--3096.

\bibitem{K}
A. W. Knapp, Local Langlands correspondence: the archimedean case,
Proc. Symp. Pure. Math. 55, 1994, 393--410.

\bibitem{La}
V. Lafforgue, $K$-th\'{e}orie bivariante pour les alg\`{e}bres de
Banach et conjecture de Baum-Connes, Invent. Math. 149 (2002)
1--95.

\bibitem{M}
C. M{\oe}glin, {Representations of $\GL(n)$ over the Real Field}, Proc. Symp. Pure Math. 61
(1997), 157--166.

\bibitem{MP}
S. Mendes, R.J. Plymen, Base change and $K$-theory for $\GL(n)$,
J. Noncommut. Geom., 1 (2007), 311--331.


\bibitem{Pl1}
R. J. Plymen, {The reduced C*-algebra of the p-adic group
$\GL(n)$}, J. Functional Analysis 72 (1987) 1--12.

\bibitem{PP}
M. G. Penington, R. J. Plymen, The Dirac operator and the principal
series for complex semisimple Lie groups, J. Functional Analysis, 53
(1983) 269--286.




\end{thebibliography}
\end{document}